\pdfoutput=1

\documentclass[10pt,oneside,reqno]{amsart}

\usepackage{epsfig}
\usepackage{amsthm}
\usepackage{amssymb}
\usepackage{amsmath}
\usepackage{amscd}
\usepackage{color}
\usepackage{esint}
\usepackage{enumitem}
\usepackage{bbm}
\usepackage{mathrsfs}
\usepackage[top=1.1in, bottom=1in, left=1in, right=1in]{geometry}

\usepackage[dvipsnames]{xcolor}
\usepackage[colorlinks=true, pdfstartview=FitV, linkcolor=blue, citecolor=blue, urlcolor=blue]{hyperref}

\newcommand{\lb}{\varLambda}

\def \<{\langle}
\def \>{\rangle}

\newcommand{\bg}{\begin{equation}}
\newcommand{\ed}{\end{equation}}
\newcommand{\bga}{\begin{eqnarray}}
\newcommand{\eda}{\end{eqnarray}}

\def\cbdu{\par{\raggedleft$\Box$\par}}

\newtheorem {Theorem}  {Theorem}

\numberwithin{Theorem}{section}

\newtheorem {Lemma}[Theorem]  {Lemma}

\theoremstyle{definition}

\theoremstyle{remark}

\def \l{\lambda}
\expandafter\chardef\csname pre amssym.def
at\endcsname=\the\catcode`\@ \catcode`\@=11
\def\undefine#1{\let#1\undefined}
\def\newsymbol#1#2#3#4#5{\let\next@\relax
 \ifnum#2=\@ne\let\next@\msafam@\else
 \ifnum#2=\tw@\let\next@\msbfam@\fi\fi
 \mathchardef#1="#3\next@#4#5}
\def\mathhexbox@#1#2#3{\relax
 \ifmmode\mathpalette{}{\m@th\mathchar"#1#2#3}%
 \else\leavevmode\hbox{$\m@th\mathchar"#1#2#3$}\fi}
\def\hexnumber@#1{\ifcase#1 0\or 1\or 2\or 3\or 4\or 5\or 6\or 7\or 8\or
 9\or A\or B\or C\or D\or E\or F\fi}

\font\teneufm=eufm10 \font\seveneufm=eufm7 \font\fiveeufm=eufm5
\newfam\eufmfam
\textfont\eufmfam=\teneufm \scriptfont\eufmfam=\seveneufm
\scriptscriptfont\eufmfam=\fiveeufm

\catcode`\@=\csname pre amssym.def at\endcsname

\newcounter{remark}
\numberwithin{equation}{section}
\numberwithin{figure}{section}

\def \grad {\nabla}

\newcommand{\el}{\varepsilon}

\renewcommand{\l}{\lambda}

\renewcommand{\a}{\alpha}
\renewcommand{\th}{\theta}
\renewcommand{\b}{\beta}

\newcommand{\R}{\mathbf{R}}

\newcommand{\les}{\lesssim}

\newcommand{\Dd}{{\mathcal D}}

\newcommand{\Ee}{{\mathcal E}}
\newcommand{\Ff}{{\mathcal F}}

\newcommand{\Ss}{{\mathcal S}}

\def  \R   {{\mathbb R}}
\def  \Z   {{\mathbb Z}}

\def  \T   {{\mathbb T}}

\def  \haf  {{\frac{1}{2}}}
\def  \p   {\partial}

\def  \Dq    {\Delta_q}
\def  \Dp    {\Delta_p}
\def  \Dtp    {\widetilde{\Delta}_p}

\def  \Sqmt  {\Ss_{q-2}}
\def  \Spmt  {\Ss_{p-2}}
\def  \lqs  {\lambda^s_q}
\def  \lqts  {\lambda^{2s}_q}

\def  \lpts  {\lambda^{2s}_p}

\def  \sump  {\sum_{p=-1}^\infty}

\def  \sumq  {\sum_{q=-1}^\infty}
\def  \sumptq {\sum_{p=-1}^q}
\def  \sumpqlt {\sum_{|p-q|\leq 2}}

\def  \inttwo  {\int_{\T^2}}

\newcommand\onenorm[1]{\lVert#1\rVert_{L^1(\T^2)}}
\newcommand\twonorm[1]{\lVert#1\rVert_{L^2(\T^2)}}

\newcommand\rnorm[1]{\Vert#1\Vert_{L^r(\T^2)}}
\newcommand\Linfnorm[1]{\Vert#1\Vert_{L^\infty(\T^2)}}

\newcommand\intttwo[1]{\int_{\T^2} #1 \, dx}

\def\build#1_#2^#3{\mathrel{\mathop{\kern 0pt#1}\limits_{#2}^{#3}}}

\begin{document}
%\currannalsline{0}{2006}

\title[Local well-posedness for split fractional $2\frac12$D EMHD]{Local well-posedness for the two-and-a-half-dimensional EMHD system with split fractional dissipation}

\author [Qirui Peng]{Qirui Peng}
\address{Department of Mathematics,  University of California, Santa Barbara, Santa Barbara, CA 93117, USA}
\email{qpeng9@ucsb.edu} 

\thanks{}

%%%use \Proof instead of \begin{proof}

%%%% use \Endproof instead of \end{proof}

%%%% use \references {999} instead of \begin{thebibliography}{99}

%%%%used \Endrefs instead of \end{thebibliography}

\begin{abstract}
We study the $2\frac12$-dimensional electron magnetohydrodynamics (EMHD)
system on $\T^2$ with componentwise fractional dissipation,
\begin{align*}
    \p_t a+a_yb_x-a_xb_y&=-\lb^\a a,\\
    \p_t b-a_y\Delta a_x+a_x\Delta a_y&=-\lb^\b b,
\end{align*}
where $0<\a,\b<2$. This system is a $2\frac12$-dimensional reduction of the
magnetic equation in Hall--MHD/EMHD under the ansatz
$B=\nabla\times(ae_z)+be_z$. We prove local well-posedness for initial data
$(a_0,b_0)\in H^{s+1}(\T^2)\times H^s(\T^2)$ with $s\geq 2-\varepsilon$,
provided that
\[
    \a+\b>2.
\]
Thus neither component is required to carry a full Laplacian dissipation; the
smoothing effects of the two fractional dissipations can be combined to control
the Hall nonlinearity. The proof is based on Littlewood--Paley energy estimates,
commutator bounds, and cancellations between the leading low--high interactions.

\bigskip


\end{abstract}

\maketitle

\section{Introduction}\label{Sec_Intro}
The Hall--magnetohydrodynamics (Hall--MHD) system is a fundamental model for
plasma dynamics in regimes where the Hall effect becomes relevant at small length
scales. Compared with the classical magnetohydrodynamics equations, the Hall term
contains one additional derivative of the magnetic field and therefore produces a
genuinely quasilinear difficulty in Sobolev energy estimates. The Hall--MHD model
has been derived and studied from the viewpoint of kinetic and two-fluid plasma
models by Acheritogaray, Degond, Frouvelle and Liu \cite{ADFL11}. On the
mathematical side, a systematic local well-posedness theory for the
three-dimensional Hall--MHD equations was developed by Chae, Degond and Liu
\cite{CDL14}; see also Chae and Lee \cite{CL14} for blow-up criteria and small
smooth solutions. More recent works have pursued lower-regularity and critical
space theories, including the critical Sobolev and Besov results of Danchin and
Tan \cite{DT21,DT22}, the larger critical Besov framework of Liu and Tan
\cite{LT21}, and the optimal Sobolev local theory of Dai \cite{D21}.

The role of fractional diffusion in Hall-type systems has also attracted
considerable attention. Chae, Wan and Wu \cite{CWW15} showed that the magnetic
diffusion in the Hall--MHD equations can be weakened to a fractional one; their
argument exploits the smoothing of the dissipative term through Littlewood--Paley
and Besov estimates. Related generalized Hall--MHD models with fractional
dissipation were studied, for instance, by Dai and Liu \cite{DL22} and Ye
\cite{Ye22}. Our result is in the same spirit of relaxation as \cite{CWW15}, but in
a different reduced EMHD setting: instead of imposing a full Laplacian
dissipation on the magnetic field, we split the available dissipation between two
scalar components and require only the combined condition $\a+\b>2$.

The electron magnetohydrodynamics (EMHD) equation is the fluid-free magnetic
subsystem obtained by retaining the Hall dynamics while neglecting the ion
velocity. It is used as a model for small-scale plasma dynamics and whistler-type
phenomena; see, for example, the numerical and physical studies of Biskamp,
Schwarz and Drake \cite{BSD96} and Biskamp, Schwarz, Zeiler and Drake
\cite{BSZD99}. In PDE analysis, EMHD is challenging because the Hall nonlinearity
contains a derivative falling on the current $\nabla\times B$. This difficulty is
particularly pronounced without resistivity: Jeong and Oh proved ill-posedness
near degenerate stationary states \cite{JO22}, while later establishing
well-posedness around nonzero uniform magnetic fields \cite{JO25}. In the
presence of resistive effects, Dai studied a two-dimensional EMHD setting near a
steady state \cite{dai2023global}. Most closely related to the present paper is
the partial-resistivity theorem of Dai and Babaei \cite{dai2025well}, where the
$2\frac12$-dimensional EMHD system is shown to be locally well-posed when either
the $a$ equation or the $b$ equation carries a full Laplacian dissipation.

We consider the $2\frac12$-dimensional EMHD ansatz in which the magnetic field is
independent of the vertical variable and is written as
\[
    B=\nabla\times(ae_z)+be_z=(a_y,-a_x,b),
\]
where $a=a(x,y,t)$ is the magnetic potential and $b=b(x,y,t)$ is the vertical
magnetic component. Allowing different fractional dissipations on these two
components leads to the system
\begin{subequations}\label{eq:twohalf_EMHD}
\begin{align}
\p_t a + a_y b_x - a_x b_y &= -\lb^\a a,  \label{eq:twohalf_EMHD_1} \\
\p_t b - a_y \Delta a_x + a_x \Delta a_y &= -\lb^\b b, \label{eq:twohalf_EMHD_2}
\end{align}
on $(0,\infty)\times \T^2$. Here $\lb^s:=(-\Delta)^{\frac{s}{2}}$ is the
Fourier multiplier defined by
\[
    \Ff(\lb^s u)(k)=|k|^s\hat u(k),\qquad k\in\Z^2,
\]
and throughout the paper we assume
\[
    0<\a,\b<2.
\]
\end{subequations}

The theorem below complements the endpoint result in \cite{dai2025well}.
We do not treat the cases $(\a,\b)=(2,0)$ or $(0,2)$; instead, we consider the
genuinely fractional regime
\[
    0<\a,\b<2,
    \qquad \a+\b>2.
\]
In this regime each component has less than one full Laplacian of dissipation,
but the two dissipations together provide enough smoothing to close the local
Sobolev estimates. In this sense the main theorem relaxes the full-diffusion
requirement in the $2\frac12$-dimensional EMHD model by sharing the dissipative
gain between the magnetic potential and the vertical component.

Compared with the direct $2\frac12$-dimensional consequence of the fractional
Hall--MHD theory of Chae, Wan and Wu \cite{CWW15}, which corresponds in the
present notation to the symmetric condition $\alpha=\beta>1$, our theorem
replaces this requirement by the componentwise condition
\[
    \alpha+\beta>2.
\]
Thus the result may be viewed as a $2\frac12$-dimensional split-dissipation
refinement of their fractional-diffusion mechanism.

\begin{Theorem}\label{Thm:main}
Let the exponents $\alpha,\beta$ in the system \eqref{eq:twohalf_EMHD} satisfy
\[
0<\alpha,\beta < 2 \quad \text{and} \quad \alpha+\beta > 2.
\]
Then there exists $\varepsilon > 0$ such that, for any initial data
\[
(a_0,b_0)\in H^{s+1}(\mathbb T^2)\times H^s(\mathbb T^2)
\]
with $s \geq 2-\varepsilon$, the system \eqref{eq:twohalf_EMHD} admits a
unique local solution on $[0,T)$ for some positive time
\[
T=T(a_0,b_0)
\]
depending on the initial data.
\end{Theorem}

Let us briefly explain the main point of the proof. The natural energy level for
the $2\frac12$-dimensional EMHD system is not symmetric in $a$ and $b$: the Hall
structure suggests estimating $a$ in $H^{s+1}$ and $b$ in $H^s$. After applying
Littlewood--Paley projections, the leading low--high interactions in the two
equations cancel, while the remaining terms are controlled by interpolating
between the energy norm
\[
    \sum_{q\geq -1}\lambda_q^{2s}
    \left(\|\lb\Delta_q a\|_{L^2}^2+\|\Delta_q b\|_{L^2}^2\right)
\]
and the dissipation norm
\[
    \sum_{q\geq -1}\lambda_q^{2s}
    \left(\|\lb^{1+\frac\a2}\Delta_q a\|_{L^2}^2
    +\|\lb^{\frac\b2}\Delta_q b\|_{L^2}^2\right).
\]
The condition $\a+\b>2$ gives the positive summability margin needed in these
dyadic estimates. This is the mechanism by which the smoothing of the two
fractional dissipations is shared between the two components.

\section{Preliminaries}\label{sec:Prelim}

\subsection{Littlewood--Paley theory}\label{LP_Intro}
We recall the inhomogeneous Littlewood--Paley decomposition on the torus
$\T^2$. Let $\chi\in C_0^\infty(\R^2)$ be a nonnegative radial function such
that
\begin{equation}\label{def:chi}
    \chi(\xi)=1 \quad \text{if } |\xi|\leq \frac34,
    \qquad
    \chi(\xi)=0 \quad \text{if } |\xi|\geq 1.
\end{equation}
Define
\begin{equation}\label{def:varphi}
    \varphi(\xi):=\chi\left(\frac{\xi}{2}\right)-\chi(\xi),
\end{equation}
and, for $q\geq -1$,
\[
    \varphi_q(\xi):=
    \begin{cases}
        \varphi(\lambda_q^{-1}\xi), & q\geq 0,\\
        \chi(\xi), & q=-1,
    \end{cases}
    \qquad \lambda_q:=2^q.
\]
Then $\{\varphi_q\}_{q\geq -1}$ forms a smooth dyadic partition of unity in
Fourier space. For a distribution $u\in\mathcal D'(\T^2)$, we set
\begin{equation}\label{def:Dq}
    \Dq u(x):=\sum_{k\in\Z^2}\hat u(k)\varphi_q(k)e^{2\pi i k\cdot x}.
\end{equation}
With this convention,
\[
    u=\sum_{q=-1}^{\infty}\Dq u
\]
in the sense of distributions. We also denote the low frequency cutoff by
\[
    \Ss_q u:=\sum_{j=-1}^{q}\Delta_j u,
\]
with the convention that $\Ss_q=0$ if $q\leq -2$. Finally, we write
\[
    \widetilde{\Delta}_q u:=\Delta_{q-1}u+\Delta_q u+\Delta_{q+1}u,
\]
where $\Delta_j=0$ for $j\leq -2$.

The Sobolev norm can be equivalently described through this decomposition:
for every $s\in\R$,
\begin{equation}\label{def_Hs_norm}
    \|u\|_{H^s(\T^2)}
    \sim
    \left(\sum_{q=-1}^{\infty}\lambda_q^{2s}
    \|\Dq u\|_{L^2(\T^2)}^2\right)^{\frac12}.
\end{equation}
The implicit constants depend only on $s$ and on the choice of the cutoffs. We
shall use this characterization throughout the proof. For background on
Littlewood--Paley theory, see \cite{bahouri2011fourier,grafakos2008classical}.

We shall repeatedly use the following Bernstein inequalities.
\begin{Lemma}\label{lemma:Bernstein}
Let $d\geq 1$, $1\leq p\leq r\leq \infty$, and $m\geq 0$. Then, for all
$q\geq -1$,
\begin{equation}\label{Bern1}
    \|\nabla^m\Dq u\|_{L^r(\T^d)}
    \les
    \lambda_q^{m+d(\frac1p-\frac1r)}
    \|\Dq u\|_{L^p(\T^d)}.
\end{equation}
Moreover, for $q\geq 0$,
\begin{equation}\label{Bern2}
    \lambda_q^m\|\Dq u\|_{L^r(\T^d)}
    \les
    \|\nabla^m\Dq u\|_{L^r(\T^d)}.
\end{equation}
\end{Lemma}

We also recall Bony's paraproduct decomposition. For two distributions $u$ and
$v$ on $\T^2$,
\begin{equation}\label{lemma:Bony_formula}
    uv
    =\sum_{l=-1}^{\infty}\Ss_{l-2}u\,\Delta_l v
    +\sum_{l=-1}^{\infty}\Delta_l u\,\Ss_{l-2}v
    +\sum_{l=-1}^{\infty}\Delta_l u\,\widetilde{\Delta}_l v.
\end{equation}
In particular, applying $\Dq$ to a transport-type product gives
\begin{align}\label{eq:Bony_transport}
    \Dq(u\cdot\nabla v)
    &=\sum_{|q-l|\leq 2}\Dq\big(\Ss_{l-2}u\cdot\nabla\Delta_l v\big)
      +\sum_{|q-l|\leq 2}\Dq\big(\Delta_l u\cdot\nabla\Ss_{l-2}v\big) \notag \\
    &\quad
      +\sum_{l\geq q-2}\Dq\big(\Delta_l u\cdot\nabla\widetilde{\Delta}_l v\big).
\end{align}
These three terms correspond respectively to the low--high, high--low, and
high--high frequency interactions.

\begin{Lemma}\label{lemma:commutator_est_1}
Let $r\in(1,\infty)$ and $s>0$. Then the following commutator estimate holds:
\[
\rnorm{[\Dq,\Spmt u \lb^s] \Dp v} \les \l^{-1}_q \Linfnorm{\grad \Spmt u} \rnorm{\Dp \lb^s v}.
\]
\end{Lemma}
\begin{proof}
Let $h_q$ be the periodized convolution kernel of $\Dq$. By the kernel
representation of the Littlewood--Paley projection, we have
\[
[\Dq,\Spmt u \lb^s] \Dp v
= \int_{\T^2} h_q(y)\left(\Spmt u(x-y)-\Spmt u(x)\right)\lb^s\Dp v(x-y)\,dy.
\]
Using the mean value theorem and Young's convolution inequality, we obtain
\begin{align*}
\rnorm{[\Dq,\Spmt u \lb^s] \Dp v}
&\leq \left(\int_{\T^2}|h_q(y)| |y|\,dy\right)
\Linfnorm{\grad \Spmt u}\rnorm{\Dp \lb^s v} \\
&\les \l_q^{-1}\Linfnorm{\grad \Spmt u}\rnorm{\Dp \lb^s v}.
\end{align*}
This proves the lemma.
\end{proof}

\section{Proof of Theorem \ref{Thm:main}}\label{Proof:Thm_main}
We start by multiplying the equation \eqref{eq:twohalf_EMHD_1} by $-\l^{2s}_q \lb^2 \Dq^2 a $, integrating over space and summing over all $q \geq -1$: 
\begin{align*}
&\haf \frac{d}{dt} \sumq \l_q^{2s} \inttwo |\lb \Dq a|^2 \,dx + \sumq \l_q^{2s} \inttwo |\lb^{1+\frac{\a}{2}} \Dq a |^2 \,dx \\
=\ & - \sumq \l_q^{2s} \inttwo \Dq (a_y b_x) \lb^2 \Dq a \,dx + \sumq \l_q^{2s} \intttwo{\Dq (a_x b_y) \lb^2 \Dq a}\\
:=\ &I_1 + I_2.
\end{align*}
Similarly, multiplying \eqref{eq:twohalf_EMHD_2} by $\l^{2s}_q \Dq^2 b$ produces 
\begin{align*}
&\haf \frac{d}{dt} \sumq \l_q^{2s}\intttwo{|\Dq b|^2} + \sumq \l_q^{2s} \intttwo{|\lb^{\frac{\b}{2}} \Dq b|^2} \\
=\ &\sumq \l_q^{2s} \intttwo{\Dq (a_y \lb^2 a) \Dq b_x} - \sumq \l_q^{2s} \intttwo{\Dq (a_x \lb^2 a) \Dq b_y} \\
:=\ &J_1 + J_2.
\end{align*} 
Adding the above equations, we arrive at 
\begin{align*}
&\haf \frac{d}{dt} \sumq \lqts \left( \twonorm{\lb \Dq a}^2 + \twonorm{\Dq b}^2 \right) + \sumq \lqts \left( \twonorm{\lb^{1+\frac{\a}{2}}\Dq a}^2 + \twonorm{\lb^{\frac{\b}{2}} \Dq b}^2 \right) \\
=\ & I_1 + I_2 + J_1 + J_2.
\end{align*}
The local well-posedness of the system \eqref{eq:twohalf_EMHD} then follows from estimates on the terms $I_1$, $I_2$, $J_1$, and $J_2$. Using Bony's paraproduct decomposition, we decompose $I_1$ and $J_1$ as follows:
\begin{align*}
    I_1 = &- \sumq \sumpqlt \lqts \intttwo{\Dq (\Spmt a_y \Dp b_x) \Dq \lb^2 a} \\
    &- \sumq \sumpqlt \lqts \intttwo{\Dq (\Spmt b_x \Dp a_y) \Dq \lb^2 a} \\
    &- \sumq \sum_{p \geq q-2} \lqts \intttwo{\Dq (\Dtp a_y \Dp b_x) \Dq \lb^2 a} := I_{11} + I_{12} + I_{13}.
\end{align*}
\begin{align*}
    J_1 &= \sumq \sumpqlt \lqts \intttwo{\Dq (\Spmt a_y \Dp \lb^2 a) \Dq b_x} \\
    &+ \sumq \sumpqlt \lqts \intttwo{\Dq (\Spmt \lb^2 a \Dp a_y) \Dq b_x} \\
    &+ \sumq \sum_{p \geq q-2} \lqts \intttwo{\Dq (\Dtp a_y \Dp \lb^2 a) \Dq b_x} := J_{11} + J_{12} + J_{13}. 
\end{align*}
Moreover, we can split $I_{11}$ and $J_{11}$ further by exploiting the commutators 
\begin{align*}
    [\Dq,\, \Spmt a_y \p_x] \Dp b &= \Dq (\Spmt a_y \Dp b_x) - \Spmt a_y \Dq \Dp b_x, \\
     [\Dq,\, \Spmt a_y \lb^2] \Dp a &= \Dq (\Spmt a_y \Dp \lb^2 a) - \Spmt a_y \Dq \Dp \lb^2 a
\end{align*}
and write
\begin{align*}
I_{11} = &- \sumq \sumpqlt \lqts \intttwo{\Dq (\Spmt a_y \Dp b_x) \Dq \lb^2 a} \\
= &- \sumq \sumpqlt \lqts \intttwo{\left( [\Dq,\, \Spmt a_y \p_x]\Dp b \right) \Dq \lb^2 a} \\
&- \sumq \sumpqlt \lqts \intttwo{ \left( \Sqmt a_y  \Dq \Dp b_x \right) \Dq \lb^2 a}\\
&- \sumq \sumpqlt \lqts \intttwo{\left( \Spmt a_y - \Sqmt a_y \right) \Dq \Dp b_x \Dq \lb^2 a} := I_{111} + I_{112} + I_{113} \, ,
\end{align*}
as well as 
\begin{align*}
J_{11} = & \sumq \sumpqlt \lqts \intttwo{\Dq (\Spmt a_y \Dp \lb^2 a) \Dq b_x} \\
= & \sumq \sumpqlt \lqts \intttwo{\left( [\Dq,\, \Spmt a_y \lb^2]\Dp  a \right) \Dq b_x} \\
+ &\sumq \sumpqlt \lqts \intttwo{ \left( \Sqmt a_y  \Dq \Dp \lb^2 a \right) \Dq b_x}\\
+ &\sumq \sumpqlt \lqts \intttwo{ \left[ (\Spmt a_y - \Sqmt a_y ) \Dq \Dp \lb^2 a \right] \Dq b_x} := J_{111} + J_{112} + J_{113} \,.
\end{align*}
In view of 
\[
\sumpqlt \Dq \Dp b_x = \Dq b_x \ \ \text{and} \ \ \sumpqlt \Dq \Dp \lb^2 a = \Dq \lb^2 a,
\]
we observe that 
\begin{equation}\label{Proof:Thm_main_0}
    I_{112} + J_{112}  = 0. 
\end{equation}

We now proceed to bound the remaining terms. Using H\"older's inequality, the commutator estimate \ref{lemma:commutator_est_1}, and Bernstein's inequality, we obtain 
\begin{align*}
    |I_{111}| &\leq \sumq \sumpqlt \lqts \Linfnorm{\Spmt a_{xy}} \twonorm{\Dp b_x} \twonorm{\Dq \lb^2 a} \\
    &\les \sumq \l^{2s+2}_q \twonorm{\Dq a} \twonorm{\Dq b}  \left( \sumptq \l_p^3 \twonorm{\Dp a}  \right) \\
    &\les \sumq \l^{\th (s+1)}_q \twonorm{\Dq a}^\th \l_q^{\th s} \twonorm{\Dq b}^\th \l_q^{(1-\th)(s+1+\frac{\a}{2})} \twonorm{\Dq a}^{1-\th}  \cdot \\
    & \qquad  \l_q^{(1-\th)(s+\frac{\b}{2})}\twonorm{\Dq b}^{1-\th} \left ( \sumptq \l_p^{s+1} \twonorm{\Dp a} \l_p^{2-s} \right) \l_q^{1-\haf (1-\th)(\a+\b)},
\end{align*}
for any $\th \in (0,1)$. Since $2<\alpha+\beta<4$, we may choose $\th$ obeying
\begin{equation}\label{Proof:Thm_main_1}
0<\theta<  1-\frac{2}{\alpha+\beta} < \haf ,
\end{equation}
so that
\begin{equation}\label{Proof:Thm_main_2}
\varepsilon:=\frac12(1-\theta)(\alpha+\beta)-1>0.
\end{equation}
We take $s\geq 2-\el$.
By Young's inequality and Bernstein's inequality, we deduce
\begin{align}
|I_{111}| &\leq \frac{1}{64} \sumq \l_q^{2s} \twonorm{\lb^{1+\frac{\a}{2}} \Dq a}^2 + \frac{1}{64} \sumq \l_q^{2s} \twonorm{\lb^{\frac{\b}{2}} \Dq b}^2  \notag \\
&\ +C \sumq \l_q^{s+1} \twonorm{\Dq a} \lqs \twonorm{\Dq b} \left( \sumptq \l_p^{s+1} \twonorm{\Dp a} \l_p^{2-s} \l_q^{1-(1-\th)(\frac{\a}{2}+ \frac{\b}{2})}\right)^\frac{1}{\th} \notag \\
&\leq \frac{1}{64} \sumq \l_q^{2s} \left( \twonorm{\lb^{1+\frac{\a}{2}} \Dq a}^2+ \twonorm{\lb^{\frac{\b}{2}}\Dq b}^2 \right) \notag \\
&\ +C \sumq \l_q^{s+1} \twonorm{\Dq a} \lqs \twonorm{\Dq b} \left( \sumptq \l_p^{s+1} \twonorm{\Dp a} \l_p^{2-s-\el} \l_{q-p}^{-\el}\right)^\frac{1}{\th}. \label{Proof:Thm_main_3}
\end{align}
We estimate the last term on the right-hand side as follows:
\begin{align*}
   &\sumq \l_q^{s+1} \twonorm{\Dq a} \lqs \twonorm{\Dq b} \left( \sumptq \l_p^{s+1} \twonorm{\Dp a} \l_p^{2-s-\el} \l_{q-p}^{-\el}\right)^\frac{1}{\th} \\
   \les &\ \sumq \l_q^{s+1} \twonorm{\Dq a} \lqs \twonorm{\Dq b}  \sumptq \left(\l_p^{s+1} \twonorm{\Dp a} \right)^\frac{1}{\th} \\
   \les &\ \left( \sumq \lqts \twonorm{\lb \Dq a}^2 \right)^\haf \left( \sumq \lqts \twonorm{\Dq b}^2 \right)^\haf \left( \sumq \l_q^{2s} \twonorm{\lb \Dq a}^2 \right)^{\frac{1}{2\th}} \\
   \les &\ \left( \sumq \lqts \twonorm{\lb \Dq a}^2 \right)^{\haf + \frac{1}{2\th}} \left( \sumq \lqts \twonorm{\Dq b}^2 \right)^\haf \\
    \les &\ \left( \sumq \lqts \twonorm{\lb \Dq a}^2 \right)^{1+\frac{1}{2\th}} + \left( \sumq \lqts \twonorm{ \Dq b}^2 \right)^{1+\frac{1}{2\th}},
\end{align*}
where we used Jensen's inequality, the fact that $2-s-\el \leq 0$, and $\th < \haf$. As a consequence, we derive 
\begin{align}\label{Proof:Thm_main_4}
|I_{111}| &\leq \frac{1}{64} \sumq \l_q^{2s} \left( \twonorm{\lb^{1+\frac{\a}{2}} \Dq a}^2+ \twonorm{\lb^{\frac{\b}{2}}\Dq b}^2 \right) \notag \\
&\ \ + C \left( \sumq \lqts \left( \twonorm{\lb \Dq a}^2 + \twonorm{\Dq b}^2 \right) \right)^{1+\frac{1}{2\th}}.
\end{align}
To bound the term $I_{113}$, observe that since 
\[
\Spmt a_y - \Sqmt a_y 
\]
consists of only finitely many shells for $|p-q| \leq 2$, we have 
\begin{align*}
|I_{113}| &\les \sumq \lqts \left( \sumpqlt \Linfnorm{\Spmt a_y -\Sqmt a_y} \twonorm{\Dq \Dp b_x} \twonorm{\Dq \lb^2 a} \right) \\
&\les \sumq \l_q^{2s+5} \twonorm{\Dq a}^2 \twonorm{\Dq b} \\
&= \sumq \l_q^{\th(s+1)} \twonorm{\Dq a}^\th \l_q^{\th s} \twonorm{\Dq b}^\th \l_q^{(1-\th)s} \twonorm{\lb^{1+\frac{\a}{2}}\Dq a}^{1-\th} \cdot \\
&\ \ \  \l_q^{(1-\th)s} \twonorm{\lb^{\frac{\b}{2}}\Dq b}^{1-\th} \l_q^{s+1} \twonorm{\Dq a} \l_q^{2-s-\el}  \\
&\leq \frac{1}{64} \sumq \lqts \left( \twonorm{\lb^{1+\frac{\a}{2}} \Dq a}^2 + \twonorm{\lb^{\frac{\b}{2}} \Dq b}^2 \right) \\
&+ C \sumq \l_q^{s+1} \twonorm{\Dq a} \l_q^s \twonorm{\Dq b} \left( \l_q^{s+1} \twonorm{\Dq a} \l_q^{2-s-\el} \right)^{\frac{1}{\th}}\\
&\leq \frac{1}{64} \sumq \lqts \left( \twonorm{\lb^{1+\frac{\a}{2}} \Dq a}^2 + \twonorm{\lb^{\frac{\b}{2}} \Dq b}^2 \right) \\
&+ C \sumq \l_q^{s+1} \twonorm{\Dq a} \l_q^s \twonorm{\Dq b} \left( \sumptq \l_p^{s+1} \twonorm{\Dp a} \l_p^{2-s-\el} \l_{q-p}^{-\el}\right)^\frac{1}{\th}.
\end{align*}
The last term above is the same as the one appearing in \eqref{Proof:Thm_main_3}. Hence, by the same argument as in $I_{111}$, we obtain 
\begin{align}\label{Proof:Thm_main_5}
    |I_{113}| &\leq \frac{1}{64} \sumq \l_q^{2s} \left( \twonorm{\lb^{1+\frac{\a}{2}} \Dq a}^2+ \twonorm{\lb^{\frac{\b}{2}}\Dq b}^2 \right) \notag \\
&\ \ + C \left( \sumq \lqts \left( \twonorm{\lb \Dq a}^2 + \twonorm{\Dq b}^2 \right) \right)^{1+\frac{1}{2\th}}.
\end{align}
By an analogous estimate, 
\begin{align*}
    I_{13} &= - \sumq \sum_{p \geq q-2} \lqts \intttwo{\Dq (\Dtp a_y \Dp b_x) \Dq \lb^2 a} \\
    &= -\sump \sum_{q \geq p-2} \lpts \intttwo{\Dp (\widetilde{\Delta}_q a_y \Dq b_x) \Dp \lb^2 a}\\
    &= - \sumq \sum_{p = -1}^{q+2} \lpts \intttwo{\Dp (\widetilde{\Delta}_q a_y \Dq b_x) \Dp \lb^2 a} \\
    &\les \sumq \sum_{p = -1}^{q+2} \lpts \Linfnorm{\widetilde{\Delta}_q a_y} \twonorm{\Dq b_x} \twonorm{\Dp \lb^2 a} \\
    &\les \sumq \l^3_q \twonorm{\Dq a} \twonorm{\Dq b} \left(  \sum_{p = -1}^{q+2} \l_p^{2s+2} \twonorm{\Dp a} \right)
      \end{align*}
\begin{align} 
    &\les \sumq \l_q^{\th s} \twonorm{\lb \Dq a}^\th \l_q^{\th s} \twonorm{\Dq b} \l_q^{(1-\th)s} \twonorm{\lb^{1+\frac{\a}{2}} \Dq a}^{1-\th} \l_q^{(1-\th)s} \cdot \notag \\
    & \ \ \ \twonorm{\lb^{\frac{\b}{2}} \Dq b}^{1-\th} \sum_{p=-1}^{q+2} \l_p^{s+1} \twonorm{\Dp a} \l_q^{2-s-\el} \l_q^{-s-1} \l_p^{s+1} \notag \\
    &\les \sumq \l_q^{2s} \left( \twonorm{\lb^{1+\frac{\a}{2}} \Dq a}^2 + \twonorm{\lb^{\frac{\b}{2}} \Dq b}^2 \right) \notag \\
    &+ C \sumq \l^s_q \twonorm{\lb \Dq a} \l_q^s \twonorm{\Dq b} \left( \sum_{p = -1}^{q+2} \l_p^{s+1} \twonorm{\Dp a} \l_p^{2-s-\el} \l_{q-p}^{-s-1} \right)^{\frac{1}{\th}}. \label{Proof:Thm_main_6}
\end{align} 
There is no essential difference between the right-hand side of \eqref{Proof:Thm_main_6} and that of \eqref{Proof:Thm_main_3}. Since $-s-1 < 0$, we proceed as in the estimate of $I_{111}$ and conclude 
\begin{align}\label{Proof:Thm_main_7}
    |I_{13}| &\leq \frac{1}{64} \sumq \l_q^{2s} \left( \twonorm{\lb^{1+\frac{\a}{2}} \Dq a}^2+ \twonorm{\lb^{\frac{\b}{2}}\Dq b}^2 \right) \notag \\
&\ \ + C \left( \sumq \lqts \left( \twonorm{\lb \Dq a}^2 + \twonorm{\Dq b}^2 \right) \right)^{1+\frac{1}{2\th}}.
\end{align}

The estimate for $I_{12}$ is slightly more delicate. We first decompose $I_{12}$ as in the treatment of $I_{11}$ and write
\begin{align*}
    I_{12} &= - \sumq \sumpqlt \lqts \intttwo{\Dq \left( \Spmt b_x \Dp a_y \right) \Dq \lb^2 a } \\
    &= - \sumq \sumpqlt \lqts \intttwo{(\left[\Dq, \Spmt b_x \p_y \right] \Dp a) \Dq \lb^2 a} \\
    &\ \ \ - \sumq \sumpqlt \lqts \intttwo{\Sqmt b_x (\Dq \Dp a_y) \Dq \lb^2 a} \\
    &\ \ \ - \sumq \sumpqlt \lqts \intttwo{\left( \Spmt b_x -\Sqmt b_x \right) (\Dq \Dp a_y) \Dq \lb^2 a} := I_{121} + I_{122} + I_{123}.
\end{align*}
Using H\"older's, Jensen's, and Young's inequalities, as well as the commutator estimate \ref{lemma:commutator_est_1}, we deduce 
\begin{align*}
    |I_{121}| &\leq \sumq \sumpqlt \l_q^{2s+2} \twonorm{\Dq a}^2 \Linfnorm{\grad \Spmt b_x} \\
    &\les \sumq \sumpqlt \l_q^{2s+2} \twonorm{\Dq a}^2 \left( \sumptq \l_p^3 \twonorm{\Dp b} \right) \\
    &=\sumq \l_q^{\th(s+1)} \twonorm{\Dq a}^\th \l_q^{(1-\th)(s+1+\frac{\a}{2})} \twonorm{\Dq a}^{1-\th} \cdot \\
    & \ \ \ \l_q^{s+1} \twonorm{\Dq a} \left( \sumptq \l_p^{\th s} \twonorm{\Dp b}^\th \l_p^{(1-\th)(s+\frac{\b}{2})} \twonorm{\Dp b}^{1-\th} \l_p^{3-s} \l_q^{-(1-\th)\frac{\a}{2}} \l_p^{-(1-\th)\frac{\b}{2}} \right) \\
    &\les \left( \sumq \l_q^{2s} \twonorm{\lb \Dq a}^2 \right)^{\frac{\th}{2}} \left( \sumq \l_q^{2s} 
    \twonorm{\lb^{1+\frac{\a}{2}} \Dq a }^2 \right)^{\frac{1-\th}{2}} \cdot \\
    & \ \ \ \left( \sumq \lqts \twonorm{\lb \Dq a}^2 \bigg( \sumptq \l_p^{\th s} \twonorm{\Dp b}^\th \l_p^{(1-\th)(s+ \frac{\b}{2})} \twonorm{\Dp b}^{1-\th} \l_p^{2-s-\el} \l_{q-p}^{-(1-\th)\frac{\a}{2}}  \bigg)^2 \right)^\haf 
\end{align*}

\begin{align*}
 &\leq \frac{1}{64} \sumq \l_q^{2s} \twonorm{\lb^{1+\frac{\a}{2}}\Dq a}^2 + C\bigg( \sumq \l_q^{2s} \twonorm{\lb \Dq a}^2 \bigg)^{\frac{\th}{1+\th}} \cdot \\
    &\ \ \ \left( \sumq \l_q^{2s} \twonorm{\lb \Dq a}^2 \bigg( \sumptq \l_p^{2\th s} \twonorm{\Dp b}^{2\th} \l_p^{2(1-\th)s} \twonorm{\lb^{\frac{\b}{2}}\Dp b}^{2(1-\th)} \l_{q-p}^{-(1-\th)\a} \bigg) \right)^{\frac{1}{1+\th}} \\
    &\leq \frac{1}{64} \sumq \l_q^{2s} \twonorm{\lb^{1+\frac{\a}{2}}\Dq a}^2 + C\bigg( \sumq \l_q^{2s} \twonorm{\lb \Dq a}^2 \bigg)^{\frac{\th}{1+\th}} \cdot \\
    &\ \ \ \bigg( \sumq \l_q^{2s} \twonorm{\lb \Dq a}^2 \bigg)^{\frac{1}{1+\th}} \left( \sumq \bigg( \sumptq \l_p^{2\th s} \twonorm{\Dp b}^{2\th} \l_p^{2(1-\th)s} \twonorm{\lb^{\frac{\b}{2}}\Dp b}^{2(1-\th)} \l_{q-p}^{-(1-\th)\a} \bigg) \right)^{\frac{1}{1+\th}} \\
     &\leq \frac{1}{64} \sumq \l_q^{2s} \twonorm{\lb^{1+\frac{\a}{2}}\Dq a}^2 \\
      &\ \ \ + C\bigg( \sumq \l_q^{2s} \twonorm{\lb \Dq a}^2 \bigg) \left( \sumq \l_q^{2 s} \twonorm{\Dq b}^{2} \right)^{\frac{\th}{1+\th}} \left( \sump \l_p^{2s} \twonorm{\lb^{\frac{\b}{2}}\Dp b} \right)^{\frac{1-\th}{1+\th}} \\
      &\leq \frac{1}{64} \sumq \l_q^{2s} \Big( \twonorm{\lb^{1+\frac{\a}{2}}\Dq a}^2 + \twonorm{\lb^{\frac{\b}{2}}\Dp b}^2 \Big) \\
      &\ \ \ + C\bigg( \sumq \l_q^{2s} \twonorm{\lb \Dq a}^2 \bigg)^{\haf + \frac{1}{2\th}} \bigg( \sumq \l_q^{2 s} \twonorm{\Dq b}^{2} \bigg)^{\frac{1}{2}}. 
\end{align*}
Therefore, we obtain the same bound for $I_{121}$ as before:
\begin{align}\label{Proof:Thm_main_8}
     |I_{121}| &\leq \frac{1}{64} \sumq \l_q^{2s} \left( \twonorm{\lb^{1+\frac{\a}{2}} \Dq a}^2+ \twonorm{\lb^{\frac{\b}{2}}\Dq b}^2 \right) \notag \\
&\ \ + C \left( \sumq \lqts \left( \twonorm{\lb \Dq a}^2 + \twonorm{\Dq b}^2 \right) \right)^{1+\frac{1}{2\th}}. 
\end{align}
The estimate for the term $I_{123}$ follows from a similar reasoning as in $I_{113}\,$: 
\begin{align*}
    |I_{123}| &\les \sumq \lqts \left( \sumpqlt \Linfnorm{\Sqmt b_x - \Spmt b_x} \twonorm{\Dq a_y} \twonorm{\Dp \lb^2 a} \right) \\
    &\les \sumq \l_q^{2s+5} \twonorm{\Dq a}^2 \twonorm{\Dq b},
\end{align*}
which yields
\begin{align}\label{Proof:Thm_main_9}
     |I_{123}| &\leq \frac{1}{64} \sumq \l_q^{2s} \left( \twonorm{\lb^{1+\frac{\a}{2}} \Dq a}^2+ \twonorm{\lb^{\frac{\b}{2}}\Dq b}^2 \right) \notag \\
&\ \ + C \left( \sumq \lqts \left( \twonorm{\lb \Dq a}^2 + \twonorm{\Dq b}^2 \right) \right)^{1+\frac{1}{2\th}}. 
\end{align}
It remains to estimate $I_{122}$. Integrating by parts, we write
\begin{align*}
    I_{122} &= - \sumq \sumpqlt \l_q^{2s} \intttwo{\Sqmt b_x (\Dq \Dp a_y) \Dq \lb^2 a} \\
    &= \sumq \lqts \intttwo{\Sqmt b_x (\Dq a_y) \left( \Dq a_{xx} + \Dq a_{yy} \right) } \\
    &= - \sumq \lqts \intttwo{\Sqmt b_{xx} (\Dq a_y) \Dq a_{x} } -\sumq \lqts \intttwo{\Sqmt b_x (\Dq a_{xy}) \Dq a_x } \\
    &\ \ \ - \sumq \lqts \intttwo{\Sqmt b_{xy} (\Dq a_y) \Dq a_{y} }-\sumq \lqts \intttwo{\Sqmt b_x (\Dq a_{yy}) \Dq a_y } \\
    &:= I_{1221} + I_{1222} + I_{1223} + I_{1224}.
\end{align*}
We can bound $I_{1221}$ by 
\begin{align*}
    |I_{1221}| &\leq \sumq \lqts \Linfnorm{\Sqmt b_{xx}} \twonorm{\Dq a_y} \twonorm{\Dq a_x} \\
    &\les \sumq \l_q^{2s+2} \twonorm{\Dq a}^2 \left( \sum_{p \leq q-2} \l_p^3 \twonorm{\Dp b} \right).
\end{align*}
The right-hand side above is the same as that of $I_{121}$; thus we conclude 
\begin{align*}
     |I_{1221}| &\leq \frac{1}{64} \sumq \l_q^{2s} \left( \twonorm{\lb^{1+\frac{\a}{2}} \Dq a}^2+ \twonorm{\lb^{\frac{\b}{2}}\Dq b}^2 \right) \notag \\
&\ \ + C \left( \sumq \lqts \left( \twonorm{\lb \Dq a}^2 + \twonorm{\Dq b}^2 \right) \right)^{1+\frac{1}{2\th}}. 
\end{align*}
For $I_{1222}$, we can carry out the estimate as follows:
\begin{align*}
    I_{1222} &= -\haf \sumq \lqts \intttwo{\Sqmt b_x \p_y (\Dq a_x)^2} \\
    &= \haf \sumq \lqts \intttwo{\Sqmt b_{xy} |\Dq a_x|^2} \\
    &\les \sumq \lqts \Linfnorm{\Sqmt b_{xy}} \twonorm{\Dq a_x}^2 \\
    &\les \sumq \l_q^{2s+2} \twonorm{\Dq a}^2 \left( \sum_{p \leq q-2} \l_p^3 \twonorm{\Dp b} \right)
\end{align*}
and consequently,
\begin{align*}
     |I_{1222}| &\leq \frac{1}{64} \sumq \l_q^{2s} \left( \twonorm{\lb^{1+\frac{\a}{2}} \Dq a}^2+ \twonorm{\lb^{\frac{\b}{2}}\Dq b}^2 \right) \notag \\
&\ \ + C \left( \sumq \lqts \left( \twonorm{\lb \Dq a}^2 + \twonorm{\Dq b}^2 \right) \right)^{1+\frac{1}{2\th}}. 
\end{align*}
The terms $I_{1223}$ and $I_{1224}$ are treated in the same way as $I_{1221}$ and $I_{1222}$, respectively. Hence we arrive at 
\begin{align}\label{Proof:Thm_main_10}
     |I_{122}| &\leq \frac{1}{64} \sumq \l_q^{2s} \left( \twonorm{\lb^{1+\frac{\a}{2}} \Dq a}^2+ \twonorm{\lb^{\frac{\b}{2}}\Dq b}^2 \right) \notag \\
&\ \ + C \left( \sumq \lqts \left( \twonorm{\lb \Dq a}^2 + \twonorm{\Dq b}^2 \right) \right)^{1+\frac{1}{2\th}}. 
\end{align}

Combining \eqref{Proof:Thm_main_8}, \eqref{Proof:Thm_main_9} and \eqref{Proof:Thm_main_10}, and redistributing the harmless numerical constants, we have
\begin{align}\label{Proof:Thm_main_11}
     |I_{12}| &\leq \frac{1}{64} \sumq \l_q^{2s} \left( \twonorm{\lb^{1+\frac{\a}{2}} \Dq a}^2+ \twonorm{\lb^{\frac{\b}{2}}\Dq b}^2 \right) \notag \\
&\ \ + C \left( \sumq \lqts \left( \twonorm{\lb \Dq a}^2 + \twonorm{\Dq b}^2 \right) \right)^{1+\frac{1}{2\th}}. 
\end{align}
We now estimate the remaining parts in the decomposition of $J_1$. By the commutator estimate \ref{lemma:commutator_est_1}, H\"older's inequality and Bernstein's inequality, 
\begin{align*}
    |J_{111}| &\leq \sumq \sumpqlt \lqts \twonorm{[\Dq,\Spmt a_y \lb^2]\Dp a}\twonorm{\Dq b_x} \\
    &\les \sumq \sumpqlt \l_q^{2s} \l_q^{-1}\Linfnorm{\grad \Spmt a_y}\twonorm{\Dp \lb^2 a}\twonorm{\Dq b_x} \\
    &\les \sumq \l_q^{2s+2}\twonorm{\Dq a}\twonorm{\Dq b}\left( \sumptq \l_p^3\twonorm{\Dp a}\right).
\end{align*}
The right-hand side above is exactly the same as the one obtained in the estimate of $I_{111}$. Therefore, by the same argument leading to \eqref{Proof:Thm_main_4},
\begin{align}\label{Proof:Thm_main_12}
    |J_{111}| &\leq \frac{1}{64} \sumq \l_q^{2s} \left( \twonorm{\lb^{1+\frac{\a}{2}} \Dq a}^2+ \twonorm{\lb^{\frac{\b}{2}}\Dq b}^2 \right) \notag \\
&\ \ + C \left( \sumq \lqts \left( \twonorm{\lb \Dq a}^2 + \twonorm{\Dq b}^2 \right) \right)^{1+\frac{1}{2\th}}.
\end{align}
Similarly, since $\Spmt a_y-\Sqmt a_y$ consists of only finitely many shells for $|p-q|\leq 2$, we have
\begin{align*}
    |J_{113}| &\les \sumq \lqts \left( \sumpqlt \Linfnorm{\Spmt a_y-\Sqmt a_y}\twonorm{\Dq\Dp \lb^2 a}\twonorm{\Dq b_x}\right) \\
    &\les \sumq \l_q^{2s+5}\twonorm{\Dq a}^2\twonorm{\Dq b}.
\end{align*}
This is the same upper bound as in the estimate of $I_{113}$; hence
\begin{align}\label{Proof:Thm_main_13}
    |J_{113}| &\leq \frac{1}{64} \sumq \l_q^{2s} \left( \twonorm{\lb^{1+\frac{\a}{2}} \Dq a}^2+ \twonorm{\lb^{\frac{\b}{2}}\Dq b}^2 \right) \notag \\
&\ \ + C \left( \sumq \lqts \left( \twonorm{\lb \Dq a}^2 + \twonorm{\Dq b}^2 \right) \right)^{1+\frac{1}{2\th}}.
\end{align}
The estimate for $J_{12}$ is more straightforward, since no integration by
parts is needed. By H\"older's inequality and Bernstein's inequality, we have
\begin{align*}
    |J_{12}|
    &\leq \sumq \sumpqlt \lqts
    \twonorm{\Dq(\Spmt \lb^2 a \Dp a_y)}
    \twonorm{\Dq b_x} \\
    &\les \sumq \sumpqlt \lqts
    \Linfnorm{\Spmt \lb^2 a}
    \twonorm{\Dp a_y}
    \twonorm{\Dq b_x} \\
    &\les \sumq
    \l_q^{2s+2}
    \twonorm{\Dq a}
    \twonorm{\Dq b}
    \left( \sumptq \l_p^3 \twonorm{\Dp a} \right).
\end{align*}
The right-hand side is the same as the one appearing in the estimate of
$I_{111}$. Indeed, using the same choice of $\th$ and $\el$ as in
\eqref{Proof:Thm_main_1} and \eqref{Proof:Thm_main_2}, we obtain
\begin{align*}
    |J_{12}|
    &\les \sumq
    \l_q^{\th(s+1)} \twonorm{\Dq a}^{\th}
    \l_q^{\th s} \twonorm{\Dq b}^{\th}
    \l_q^{(1-\th)(s+1+\frac{\a}{2})}
    \twonorm{\Dq a}^{1-\th} \cdot \\
    &\qquad
    \l_q^{(1-\th)(s+\frac{\b}{2})}
    \twonorm{\Dq b}^{1-\th}
    \left( \sumptq \l_p^{s+1}\twonorm{\Dp a}\l_p^{2-s} \right)
    \l_q^{1-\haf(1-\th)(\a+\b)}.
\end{align*}
Therefore, by Young's inequality and the same summation argument as in
$I_{111}$, we conclude
\begin{align}\label{Proof:Thm_main_J12}
    |J_{12}|
    &\leq
    \frac{1}{64}
    \sumq \l_q^{2s}
    \left(
        \twonorm{\lb^{1+\frac{\a}{2}}\Dq a}^2
        +
        \twonorm{\lb^{\frac{\b}{2}}\Dq b}^2
    \right) \notag \\
    &\ \ +
    C\left(
        \sumq \lqts
        \left(
            \twonorm{\lb \Dq a}^2
            +
            \twonorm{\Dq b}^2
        \right)
    \right)^{1+\frac{1}{2\th}}.
\end{align}
It remains to bound $J_{13}$. Using Bernstein's inequality in the form
\[
    \twonorm{\Dq f}\les \l_q\onenorm{f},
\]
we have
\begin{align*}
    |J_{13}| &\leq \sumq\sum_{p\geq q-2}\lqts\twonorm{\Dq(\Dtp a_y\Dp\lb^2 a)}\twonorm{\Dq b_x} \\
    &\les \sumq\sum_{p\geq q-2}\l_q^{2s+2}\l_p^3\twonorm{\Dp a}^2\twonorm{\Dq b}.
\end{align*}
For $p\geq q-2$, using \eqref{Proof:Thm_main_2} again, we may write
\begin{align*}
    |J_{13}| &\les \sumq\sum_{p\geq q-2}
    \l_p^{\th(s+1)}\twonorm{\Dp a}^{\th}\l_q^{\th s}\twonorm{\Dq b}^{\th}
    \l_p^{(1-\th)(s+1+\frac{\a}{2})}\twonorm{\Dp a}^{1-\th}\cdot \\
    &\qquad \l_q^{(1-\th)(s+\frac{\b}{2})}\twonorm{\Dq b}^{1-\th}
    \l_p^{s+1}\twonorm{\Dp a}\l_p^{2-s-\el}\l_{p-q}^{-\el}.
\end{align*}
Indeed, the remaining power is summable because $2-s-\el\leq0$ and
$\{\l_{p-q}^{-\el}\}_{p\ge q-2}\in \ell^1$. Applying Young's inequality and
then the discrete Young convolution inequality, we get
\begin{align*}
    |J_{13}| &\leq \frac{1}{64}\sumq \lqts\left(\twonorm{\lb^{1+\frac{\a}{2}}\Dq a}^2+\twonorm{\lb^{\frac{\b}{2}}\Dq b}^2\right) \\
    &\quad +C\sumq\sum_{p\geq q-2}\l_p^{s+1}\twonorm{\Dp a}\l_q^s\twonorm{\Dq b}
    \left(\l_p^{s+1}\twonorm{\Dp a}\l_p^{2-s-\el}\l_{p-q}^{-\el}\right)^{\frac{1}{\th}} \\
    &\leq \frac{1}{64}\sumq \lqts\left(\twonorm{\lb^{1+\frac{\a}{2}}\Dq a}^2+\twonorm{\lb^{\frac{\b}{2}}\Dq b}^2\right) \\
    &\quad + C\left(\sumq \lqts\twonorm{\lb\Dq a}^2\right)^{\haf+\frac{1}{2\th}}
    \left(\sumq \lqts\twonorm{\Dq b}^2\right)^{\haf}.
\end{align*}
Therefore,
\begin{align}\label{Proof:Thm_main_18}
    |J_{13}| &\leq \frac{1}{64} \sumq \l_q^{2s} \left( \twonorm{\lb^{1+\frac{\a}{2}} \Dq a}^2+ \twonorm{\lb^{\frac{\b}{2}}\Dq b}^2 \right) \notag \\
&\ \ + C \left( \sumq \lqts \left( \twonorm{\lb \Dq a}^2 + \twonorm{\Dq b}^2 \right) \right)^{1+\frac{1}{2\th}}.
\end{align}
Combining \eqref{Proof:Thm_main_0}, \eqref{Proof:Thm_main_4}, \eqref{Proof:Thm_main_5}, \eqref{Proof:Thm_main_7}, \eqref{Proof:Thm_main_11}, \eqref{Proof:Thm_main_12}, \eqref{Proof:Thm_main_13}, \eqref{Proof:Thm_main_J12} and \eqref{Proof:Thm_main_18}, and again absorbing harmless numerical constants, we arrive at
\begin{align}\label{Proof:Thm_main_19}
    |I_1+J_1| &\leq \frac{1}{4} \sumq \l_q^{2s} \left( \twonorm{\lb^{1+\frac{\a}{2}} \Dq a}^2+ \twonorm{\lb^{\frac{\b}{2}}\Dq b}^2 \right) \notag \\
&\ \ + C \left( \sumq \lqts \left( \twonorm{\lb \Dq a}^2 + \twonorm{\Dq b}^2 \right) \right)^{1+\frac{1}{2\th}}.
\end{align}
The terms $I_2$ and $J_2$ can be treated in the same way, with the roles of
$x$ and $y$ interchanged. Thus
\begin{align}\label{Proof:Thm_main_20}
    |I_2+J_2| &\leq \frac{1}{4} \sumq \l_q^{2s} \left( \twonorm{\lb^{1+\frac{\a}{2}} \Dq a}^2+ \twonorm{\lb^{\frac{\b}{2}}\Dq b}^2 \right) \notag \\
&\ \ + C \left( \sumq \lqts \left( \twonorm{\lb \Dq a}^2 + \twonorm{\Dq b}^2 \right) \right)^{1+\frac{1}{2\th}}.
\end{align}
Substituting \eqref{Proof:Thm_main_19} and \eqref{Proof:Thm_main_20} into the energy identity, we obtain
\begin{align*}
&\frac{d}{dt}\sumq \lqts\left(\twonorm{\lb\Dq a}^2+\twonorm{\Dq b}^2\right)
+\sumq \lqts\left(\twonorm{\lb^{1+\frac{\a}{2}}\Dq a}^2+\twonorm{\lb^{\frac{\b}{2}}\Dq b}^2\right) \\
&\qquad \leq C\left(\sumq \lqts\left(\twonorm{\lb\Dq a}^2+\twonorm{\Dq b}^2\right)\right)^{1+\frac{1}{2\th}}.
\end{align*}
Set
\begin{align*}
    \Ee_s(t)&:=\sumq \lqts\left(\twonorm{\lb\Dq a(t)}^2+\twonorm{\Dq b(t)}^2\right), \\
    \Dd_s(t)&:=\sumq \lqts\left(
        \twonorm{\lb^{1+\frac{\a}{2}}\Dq a(t)}^2
        +\twonorm{\lb^{\frac{\b}{2}}\Dq b(t)}^2
    \right).
\end{align*}
Then the previous inequality can be written as
\begin{equation}\label{Proof:Thm_main_gronwall_0}
    \frac{d}{dt}\Ee_s(t)+\Dd_s(t)
    \leq C\Ee_s(t)^{1+\gamma},
    \qquad \gamma:=\frac{1}{2\th}>0.
\end{equation}
Dropping the nonnegative dissipation term and applying a nonlinear Gronwall
argument, we obtain the local-in-time control of \(\Ee_s\). Indeed, if
\(\Ee_s(0)>0\), then
\[
    \frac{d}{dt}\Ee_s(t)^{-\gamma}
    =-\gamma \Ee_s(t)^{-\gamma-1}\frac{d}{dt}\Ee_s(t)
    \geq -\gamma C.
\]
Hence
\begin{equation}\label{Proof:Thm_main_gronwall_1}
    \Ee_s(t)
    \leq
    \frac{\Ee_s(0)}{
        \left(1-\gamma C t\,\Ee_s(0)^\gamma\right)^{1/\gamma}}
\end{equation}
for all times for which the denominator is positive. Thus, choosing for instance
\[
    T_0:=\frac{1}{2\gamma C\left(1+\Ee_s(0)\right)^\gamma},
\]
we have
\begin{equation}\label{Proof:Thm_main_gronwall_2}
    \sup_{0\leq t\leq T_0}\Ee_s(t)
    \leq 2^{1/\gamma}\Ee_s(0)
    =2^{2\th}\Ee_s(0).
\end{equation}
The case \(\Ee_s(0)=0\) follows by the same differential inequality after a
standard regularization, or simply by applying the above estimate to
\(\Ee_s(t)+\delta\) and letting \(\delta\to0\). Integrating
\eqref{Proof:Thm_main_gronwall_0} over \([0,T_0]\) and using
\eqref{Proof:Thm_main_gronwall_2}, we further get
\begin{equation}\label{Proof:Thm_main_gronwall_3}
    \sup_{0\leq t\leq T_0}\Ee_s(t)
    +\int_0^{T_0}\Dd_s(t)\,dt
    \leq C\left(\Ee_s(0)\right).
\end{equation}
This is the a priori bound used below in the construction of local solutions.

To justify the preceding a priori estimates, we use a standard Fourier--Galerkin
approximation. Let $P_N$ denote the projection onto the Fourier modes
$\{k\in\mathbb Z^2: |k|\leq N\}$, and set
\[
    a_0^N=P_Na_0,\qquad b_0^N=P_Nb_0.
\]
We consider the finite-dimensional system
\begin{align*}
    \partial_t a^N
    + P_N\big( a_y^N b_x^N-a_x^N b_y^N\big)
    &= -\lb^\alpha a^N,\\
    \partial_t b^N
    - P_N\big( a_y^N \Delta a_x^N-a_x^N \Delta a_y^N\big)
    &= -\lb^\beta b^N,
\end{align*}
with initial data $(a_0^N,b_0^N)$. This is an ordinary differential equation
on a finite-dimensional space and hence admits a smooth local solution.
Since $P_N$ is self-adjoint and commutes with $\lb$ and $\Dq$, the estimates
above apply to $(a^N,b^N)$ with constants independent of $N$. Therefore,
for some time $T>0$ depending only on
$\|a_0\|_{H^{s+1}}+\|b_0\|_{H^s}$, we have
\begin{align*}
        &\sup_{0\leq t\leq T}
    \sumq \lqts\left(
        \twonorm{\lb\Dq a^N}^2+\twonorm{\Dq b^N}^2
    \right)
    \\
   + &\int_0^T
    \sumq \lqts\left(
        \twonorm{\lb^{1+\frac{\alpha}{2}}\Dq a^N}^2
        +\twonorm{\lb^{\frac{\beta}{2}}\Dq b^N}^2
    \right)\,dt
    \leq C.
\end{align*}

By the standard compactness argument, up to a subsequence,
\[
    (a^N,b^N)\to (a,b)
\]
strongly in lower Sobolev spaces and weakly in the above energy spaces.
Passing to the limit in the Galerkin system gives a solution of
\eqref{eq:twohalf_EMHD} on $[0,T]$ with
\[
    (a,b)\in L^\infty(0,T;H^{s+1}\times H^s).
\]
Finally, applying the same energy estimate to the difference of two solutions
gives uniqueness. This completes the proof of Theorem \ref{Thm:main}.

\section*{Acknowledgments} 
The author is grateful to Xu Yang and Mimi Dai for several stimulating discussions.

\bibliographystyle{plain}
\bibliography{twohalfEMHD}

@article {CDL14,
    AUTHOR = {Chae, Dongho and Degond, Pierre and Liu, Jian-Guo},
     TITLE = {Well-posedness for {H}all-magnetohydrodynamics},
   JOURNAL = {Ann. Inst. H. Poincar\'e{} C Anal. Non Lin\'eaire},
  FJOURNAL = {Annales de l'Institut Henri Poincar\'e{} C. Analyse Non
              Lin\'eaire},
    VOLUME = {31},
      YEAR = {2014},
    NUMBER = {3},
     PAGES = {555--565},
      ISSN = {0294-1449,1873-1430},
   MRCLASS = {35Q35 (35B30 35B53 35L60 76W05)},
  MRNUMBER = {3208454},
MRREVIEWER = {Iuliana\ Oprea},
       DOI = {10.1016/j.anihpc.2013.04.006},
       URL = {https://doi.org/10.1016/j.anihpc.2013.04.006},
}

@article {CWW15,
    AUTHOR = {Chae, Dongho and Wan, Renhui and Wu, Jiahong},
     TITLE = {Local well-posedness for the {H}all-{MHD} equations with
              fractional magnetic diffusion},
   JOURNAL = {J. Math. Fluid Mech.},
  FJOURNAL = {Journal of Mathematical Fluid Mechanics},
    VOLUME = {17},
      YEAR = {2015},
    NUMBER = {4},
     PAGES = {627--638},
      ISSN = {1422-6928,1422-6952},
   MRCLASS = {35Q35 (35B35 35B65 35R11 76W05)},
  MRNUMBER = {3412271},
       DOI = {10.1007/s00021-015-0222-9},
       URL = {https://doi.org/10.1007/s00021-015-0222-9},
}

@article {ADFL11,
    AUTHOR = {Acheritogaray, Marion and Degond, Pierre and Frouvelle, Amic
              and Liu, Jian-Guo},
     TITLE = {Kinetic formulation and global existence for the
              {H}all-{M}agneto-hydrodynamics system},
   JOURNAL = {Kinet. Relat. Models},
  FJOURNAL = {Kinetic and Related Models},
    VOLUME = {4},
      YEAR = {2011},
    NUMBER = {4},
     PAGES = {901--918},
      ISSN = {1937-5093,1937-5077},
   MRCLASS = {35Q35 (35A01 35D30 35K55 35L60 76W05)},
  MRNUMBER = {2861579},
MRREVIEWER = {Adrian\ Carabineanu},
       DOI = {10.3934/krm.2011.4.901},
       URL = {https://doi.org/10.3934/krm.2011.4.901},
}

@article {D21,
    AUTHOR = {Dai, Mimi},
     TITLE = {Local well-posedness for the {H}all-{MHD} system in optimal
              {S}obolev spaces},
   JOURNAL = {J. Differential Equations},
  FJOURNAL = {Journal of Differential Equations},
    VOLUME = {289},
      YEAR = {2021},
     PAGES = {159--181},
      ISSN = {0022-0396,1090-2732},
   MRCLASS = {76D03 (35Q35 76W05)},
  MRNUMBER = {4248458},
MRREVIEWER = {Abhik\ Kumar\ Sanyal},
       DOI = {10.1016/j.jde.2021.04.019},
       URL = {https://doi.org/10.1016/j.jde.2021.04.019},
}

@article{dai2023global,
  title={Global existence of {$2D$} electron {MHD} near a steady state},
  author={Dai, Mimi},
  journal={arXiv preprint arXiv:2306.13036},
  year={2023}
}

@article{dai2025well,
  title={Well-posedness of the electron {MHD} with partial resistivity},
  author={Dai, Mimi and Babaei, Hassan},
  journal={arXiv preprint arXiv:2503.18149},
  year={2025}
}

@book{bahouri2011fourier,
  title={Fourier Analysis and Nonlinear Partial Differential Equations},
  author={Bahouri, Hajer and Chemin, Jean-Yves and Danchin, Rapha{\"e}l},
  series={Grundlehren der mathematischen Wissenschaften},
  volume={343},
  year={2011},
  publisher={Springer},
  doi={10.1007/978-3-642-16830-7}
}

@book{grafakos2008classical,
  title={Classical Fourier Analysis},
  author={Grafakos, Loukas},
  series={Graduate Texts in Mathematics},
  volume={249},
  edition={2},
  year={2008},
  publisher={Springer}
}

@article{CL14,
  author  = {Chae, Dongho and Lee, Jinkyoung},
  title   = {On the blow-up criterion and small data global existence for the {H}all-magnetohydrodynamics},
  journal = {Journal of Differential Equations},
  volume  = {256},
  number  = {11},
  pages   = {3835--3858},
  year    = {2014},
  doi     = {10.1016/j.jde.2014.03.003}
}

@article{JO22,
  author  = {Jeong, In-Jee and Oh, Sung-Jin},
  title   = {On the {C}auchy problem for the {H}all and electron magnetohydrodynamic equations without resistivity {I}: Illposedness near degenerate stationary solutions},
  journal = {Annals of PDE},
  volume  = {8},
  number  = {2},
  pages   = {15},
  year    = {2022},
  doi     = {10.1007/s40818-022-00134-5}
}

@article{JO25,
  author  = {Jeong, In-Jee and Oh, Sung-Jin},
  title   = {Wellposedness of the Electron {MHD} Without Resistivity for Large Perturbations of the Uniform Magnetic Field},
  journal = {Annals of PDE},
  volume  = {11},
  pages   = {14},
  year    = {2025},
  doi     = {10.1007/s40818-025-00198-z}
}

@article {DT21,
    AUTHOR = {Danchin, Rapha{\"e}l and Tan, Jin},
     TITLE = {On the well-posedness of the
              {H}all-magnetohydrodynamics system in critical spaces},
   JOURNAL = {Comm. Partial Differential Equations},
  FJOURNAL = {Communications in Partial Differential Equations},
    VOLUME = {46},
      YEAR = {2021},
    NUMBER = {1},
     PAGES = {31--65},
       DOI = {10.1080/03605302.2020.1822392},
       URL = {https://doi.org/10.1080/03605302.2020.1822392},
}

@article {DT22,
    AUTHOR = {Danchin, Rapha{\"e}l and Tan, Jin},
     TITLE = {The global solvability of the
              {H}all-magnetohydrodynamics system in critical {S}obolev spaces},
   JOURNAL = {Commun. Contemp. Math.},
  FJOURNAL = {Communications in Contemporary Mathematics},
    VOLUME = {24},
      YEAR = {2022},
    NUMBER = {10},
     PAGES = {Paper No. 2150099},
}

@article {LT21,
    AUTHOR = {Liu, Lvqiao and Tan, Jin},
     TITLE = {Global well-posedness for the
              {H}all-magnetohydrodynamics system in larger critical {B}esov
              spaces},
   JOURNAL = {J. Differential Equations},
  FJOURNAL = {Journal of Differential Equations},
    VOLUME = {274},
      YEAR = {2021},
     PAGES = {382--413},
       DOI = {10.1016/j.jde.2020.10.014},
       URL = {https://doi.org/10.1016/j.jde.2020.10.014},
}

@article {DL22,
    AUTHOR = {Dai, Mimi and Liu, Han},
     TITLE = {On well-posedness of generalized
              {H}all-magneto-hydrodynamics},
   JOURNAL = {Z. Angew. Math. Phys.},
  FJOURNAL = {Zeitschrift f{"u}r angewandte Mathematik und Physik},
    VOLUME = {73},
      YEAR = {2022},
    NUMBER = {4},
     PAGES = {Paper No. 139},
       DOI = {10.1007/s00033-022-01771-3},
       URL = {https://doi.org/10.1007/s00033-022-01771-3},
}

@article {Ye22,
    AUTHOR = {Ye, Zhuan},
     TITLE = {Well-posedness results for the {$3D$} incompressible
              {H}all-{MHD} equations},
   JOURNAL = {J. Differential Equations},
  FJOURNAL = {Journal of Differential Equations},
    VOLUME = {321},
      YEAR = {2022},
     PAGES = {130--216},
       DOI = {10.1016/j.jde.2022.03.012},
       URL = {https://doi.org/10.1016/j.jde.2022.03.012},
}

@article {BSD96,
    AUTHOR = {Biskamp, D. and Schwarz, E. and Drake, J. F.},
     TITLE = {Two-dimensional electron magnetohydrodynamic turbulence},
   JOURNAL = {Phys. Rev. Lett.},
  FJOURNAL = {Physical Review Letters},
    VOLUME = {76},
      YEAR = {1996},
    NUMBER = {8},
     PAGES = {1264--1267},
       DOI = {10.1103/PhysRevLett.76.1264},
       URL = {https://doi.org/10.1103/PhysRevLett.76.1264},
}

@article {BSZD99,
    AUTHOR = {Biskamp, D. and Schwarz, E. and Zeiler, A. and Drake, J. F.},
     TITLE = {Electron magnetohydrodynamic turbulence},
   JOURNAL = {Phys. Plasmas},
  FJOURNAL = {Physics of Plasmas},
    VOLUME = {6},
      YEAR = {1999},
    NUMBER = {3},
     PAGES = {751--758},
       DOI = {10.1063/1.873312},
       URL = {https://doi.org/10.1063/1.873312},
}

\end{document}